\documentclass[11pt]{article}
\usepackage[utf8]{inputenc}
\usepackage{amsfonts, amsmath}
\usepackage{amssymb}
\usepackage{amsthm,amsmath,amscd}
\usepackage{enumerate}
\usepackage{url}
\usepackage[pdftex]{graphicx}
\usepackage[margin=25pt,font=small,labelfont=bf]{caption}
\usepackage{subfig}
\usepackage[numbers,square,sort&compress]{natbib}
\usepackage[colorlinks=true]{hyperref}
\newcommand{\defn}[1]{\textcolor{blue}{\emph{#1}}}
\newcommand*{\doi}[1]{doi: \href{https://dx.doi.org/#1}{\urlstyle{rm}\nolinkurl{#1}}}
\newcommand*{\arxiv}[1]{arXiv:  \href{https://arxiv.org/abs/#1}{\urlstyle{rm}\nolinkurl{#1}}}
\let\oldproofname=\proofname
\renewcommand{\proofname}{\rm\bf{\oldproofname}}

\newcommand{\RR}{\mathbb R}
\newcommand{\CC}{\mathbb C}

\newcommand{\QQ}{\mathbb Q}
\newcommand{\QQQ}{\mathcal Q}

\newcommand{\R}{\mathbb R}

\newcommand{\bna}{\begin{eqnarray}}
\newcommand{\ena}{\end{eqnarray}}
\newcommand{\ba}{\begin{eqnarray*}}
\newcommand{\ea}{\end{eqnarray*}}
\newcommand{\bs}[1]{}
\newcommand{\iprod}[2]{\left\langle {#1}, {#2}\right\rangle}
\newcommand{\f}{{\mathbf f}}
\newcommand{\corank}{\operatorname{corank}}
\newcommand{\stressedCorank}{\operatorname{stressedCorank}}
\newcommand{\eps}{\varepsilon}

\newtheorem{theorem}{Theorem}[section]
\newtheorem{corollary}[theorem]{Corollary}
\newtheorem{lemma}[theorem]{Lemma}
\newtheorem{proposition}[theorem]{Proposition}

\newtheorem{definition}[theorem]{Definition}

\newcommand{\ra}{\rangle}
\newcommand{\la}{\langle}

\newcommand{\ST}{{\rm Str}}
\newcommand{\GST}{{\rm Gstr}}
\newcommand{\FST}{{\rm Fstr}}
\newcommand{\PGST}{{\rm PGstr}}
\newcommand{\GOR}{{\rm GOR}}

\newcommand{\LL}{{\rm LSS}}
\newcommand{\Rank}{{\rm Rank}}
\newcommand{\CGOR}{{\rm CGOR}}

\def\p{{\bf p}}
\def\e{{\bf e}}
\def\q{{\bf q}}
\def\0{{\bf 0}}

\let\oldv=\v
\def\v{{\bf v}}
\def\w{{\bf w}}
\def\e{{\bf e}}

\def\x{{\bf x}}

\usepackage[margin=1in]{geometry}

\begin{document}
\title{
General Position Stresses}

\author{Robert Connelly
\and
Steven J. Gortler
\and
Louis Theran}
\date{}
\maketitle

%%%%%%%%%%%%%%%%
\begin{abstract}
Let $G$ be a graph with $n$ vertices, and $d$ be a
target dimension.
In this paper we study the set of
rank $n-d-1$ matrices that are
equilibrium stress matrices
for at least one
(unspecified) $d$-dimensional framework of $G$ in
general position.
In particular, we show that this set is
algebraically irreducible.
Likewise, we show that the set of frameworks with such
equilibrium stress matrices is irreducible.
As an application, this leads to a new and direct
proof that
every generically globally rigid graph has a generic
framework that is universally rigid.
\end{abstract}
%%%%%%%%%%%%%%%%
\section{Introduction} \label{section:introduction}
%%%%%%%%%%%%%%%

Equilibrium stresses are an essential tool in the study of graph rigidity.
Stresses can carry information about local, global and universal rigidity for specific and generic frameworks of a graph.
An equilibrium stress of a framework $(G,\p)$ is a vector
$\omega$ in the cokernel of its rigidity matrix $R(\p)$.
Once $\p$ is fixed, the equilibrium stresses are given by the solutions to
\begin{equation}\label{eq: equil1}
     \omega^t R(\p) = 0
\end{equation}
and form a linear space.  The dimension $s$ of the space of stresses
determines static, and so kinematic, rigidity of $(G,\p)$ via the
Maxwell index theorem
\begin{equation}\label{eq: maxwell-index}
    s - f  = m - dn + \binom{d+1}{2}
\end{equation}
where $f$ is the dimension of the space of non-trivial infinitesimal
flexes\footnote{Formally, this is the quotient of the space
of all flexes by the subspace of trivial flexes.}
and $n$ and $m$ are the number of vertices and edges of the graph $G$.

On the other hand, one can fix a set of edge weights $\omega$
and reformulate the equilibrium condition \eqref{eq: equil1}
as the vector equation
\begin{equation}\label{eq: equil2}
    \sum_{j\neq i} \omega_{ij}(\p_i - \p_j) = 0
    \qquad \text{(at all vertices $i$)}
\end{equation}
where we have $\omega_{ij} = \omega_{ji}$ and
$\omega_{ij} = 0$ for non-edges.  Holding $\omega$ fixed and
treating $\p$ as variable, the l.h.s. of \eqref{eq: equil2}
defines a linear map on $1$-dimensional configurations that
has as its matrix $\Omega$ a symmetric matrix with zeros
on the non-edges and the all ones vector in its kernel.

In recent work on global and universal rigidity \cite{Gortler-Thurston, Gortler-thurston2, ghave},
it has been useful instead to let $\p$ vary and
study the variety of all the equilibrium stresses that a graph $G$
in $d$-dimensions.
This variety is simply the set of symmetric $n$-by-$n$ matrices $\Omega$ of rank $\le n-d-1$,
that have the all ones vector in their kernel and zeros in the entries
corresponding to non edges of $G$.  

For our applications, described below,
we are most interested in knowing when this (real) ``stress variety''
is irreducible. Importantly, irreducibility implies that any subset that satisfies
an extra algebraic condition must be of strictly lower dimension. In turn, this
means that various properties can be shown to either hold almost everywhere,
or to hold almost nowhere. Similarly,
this irreducibility means that 
we can talk about  ``generic'' stresses in this variety
and how they behave.

The general question of when a linear section of a determinantal 
variety by coordinate hyperplanes is irreducbile has been 
studied in the algebro-geometric community \cite{merle, eisen, herz, conca,kumar}. 
Unfortunately for us,
many of the known irreducibility results only apply for 
matrices of very high or very low rank and, so, do not suffice 
for our applications.

% A number of related algebraic geometric questions have been previously
% studied~\cite{merle, eisen, herz, conca,kumar}. Unfortunately for us,
% many of the known irreducibility results only apply for very
% high or very low ranked matrices.

Another related question studied in the literature is the following:
given a graph $G$, which frameworks $(G,\p)$ support
an unexpectedly high-dimensional spaces of stresses? Works such as
\cite{White-Whiteley} and \cite{karpenkov} approach the problem from
the point of view of the Grassmann-Cayley algebra (see, e.g., \cite{BBR}).
These only give an implicit description, as opposed to the kind of
structural results we are after.
Universality results for stresses, such as \cite{panina}, suggest that
the limitations of the above-mentioned results arise for a fundamental
reason, namely that there are pairs $(G,d)$ for which the $d$-dimensional
stress variety of $G$ can be very complicated.  This would seem to rule out
an explicit description for all graphs.

\paragraph{Results and guide to reading}
In this paper we take a different approach.  Instead of looking 
to parameterize \emph{all} stresses supported by some $d$-dimensional 
framework, we restrict our attention to the equilibrium 
stresses for $G$ that have rank exactly $n - d - 1$
and such that any set of $n - d - 1$
columns is linearly independent.
We call these ``general position stresses'' 
(or, for short, ``Gstresses'') 
and denote the set of them by $\GST$. 
For our applications, a 
good understanding of such $\GST$ is usually enough.

Let us start with  the special case
in which $G$ is a  a $(d+1)$-connected graph 
graph that contains a $K_{d+1}$ subgraph.  A consequence 
of the ``rubber band'' theorem of Linial--Lovász--Widgerson \cite{llw}
is that we can put any generic edge weights on the edges of $G$ outside 
the distinguished $K_{d+1}$, and there will be a unique assignment of weights
to the edges of this $K_{d+1}$ so that
we will obtain a general position 
stress of rank $n - d - 1$ for $G$.  Analyzing the resulting map 
more closely, we obtain a parameterization of the general 
position stresses for $G$.
\begin{theorem}
  Let $G$ be a $(d+1)$-connected graph with $m$ edges  and suppose that $H$ is a
  $K_{d+1}$ subgraph of of $G$.
  Then $\GST$ is irreducible, of dimension $m - \binom{d+1}{2}$, and parameterized by
  a Zariski open subset of $\RR^{m - \binom{d+1}{2}}$,
  representing the weights on the edges of $G\setminus H$,
  under a rational map.
\end{theorem}
The existence of a $K_{d+1}$ subgraph is important as a hypothesis 
if we want to use rubber bands: if we were to put generic weights on 
\emph{every} edge, we will obtain an equilibrium stress matrix with 
rank $n - 1$, which correspond to frameworks with all points on 
top of each other.  On the other hand, if we pick some arbitrary set 
of $m - \binom{d+1}{2}$ edges and give them generic weights, the rubber
band construction no longer applies.

To deal with the general setting, where there may not be a $K_{d+1}$
subgraph, we will replace rubber bands by ``orthogonal representations'',
which were introduced by Lovász--Saks--Schrijver \cite{Lovasz-Schrijver},
as the basis of a parameterization of $\GST$.

Our main result is that $\GST$ is also  well behaved in this general setting.
If $G$ has $m$ edges and is $(d+1)$-connected then $\GST$ is
an irreducible, quasi-projective variety of dimension $m - \binom{d+1}{2}$.
If $G$ is not $(d+1)$-connected, then $\GST$ is empty.

As an immediate application of this result, we can obtain a new direct
proof the main result from~\cite{ghave}, namely that a graph that
is generically globally rigid in $d$-dimensions must have a generic
$d$-dimensional framework that is super stable (and thus universally rigid).

Next, we can turn things around and look at the set of $d$-dimensional frameworks
that are general position and have a stress of rank $n-d-1$.
We call these the ``Gstressable'' frameworks.
Indeed, we will see that the Gstressable frameworks
are exactly those frameworks with a Gstress.
Our main result on this topic is that the
set of Gstressable frameworks  is an irreducible constructible set of
configurations.

We will see that when $G$ is not $d+1$ connected,
the Gstressable set is empty, and when
$G$ is generically globally rigid,  this set includes almost all $d$-dimensional frameworks.
So the most interesting case is for graphs in-the-middle of these two extremes.
As we will see, for such graphs, if a  framework is Gstressable, it must be the case that
its rigidity matrix has dropped rank. For example if $G$ is generically locally
rigid, then a Gstressable framework must be infinitesimally flexible.
Still,
not every framework with a deficient rigidity matrix is  Gstressable.

Finally, we show that the requirement of general position can be relaxed
somewhat without sacrifice. In particular we define a framework as ``Fstressable''
if it has a stress of rank $n-d-1$, and each vertex has a full dimensional
affine span. We then show that any Fstressable can be approximated by
Gstressable frameworks. This places the Fstressable set in the
Euclidean and thus Zariski closure of the Gstressable set, rendering it
irreducible as well.

As mentioned above, we are most interested in the property of
irreducibility, but along the way, we will take care to prove
what we can about the algebraic object-types of the sets we encounter.
The nicest objects that we will deal with
are ``algebraic''.
Such an object is cut out  of $\RR^N$ using algebraic equalities.
The next type is ``quasi-projective''. Such an object is cut out
using algebraic
equalities and non-equalities ($\neq$).
The next larger class is ``constructible''.
This class allows for finite unions of quasi-projective objects,
or alternatively, a finite number of Boolean operations over algebraic
sets. These objects are less uniform, but when irreducible, still
contain almost all of the points in their Zariski-closure.
Finally there are the semi-algebraic sets.
Such an object is cut out
using algebraic
equalities  non-equalities ($\neq$), and inequalities.
Semi-algebraic objects
can comprise very small
portions of their Zariski-closures.

This paper relies essentially on the construction of
GORs from~\cite{Lovasz-Schrijver}, the construction of PSD stresses from GORs from
\cite{Alfakih-conn}, and the detailed study of the combination 
from~\cite{ghave}.  The main innovation of this paper is to use 
GORs in the complex setting. This will allow us to to construct and study
the full set of Gstresses (real, but of unconstrained signature).

%%%%%%%%%%%%%%%%
\section{Background}\label{sec:background}
%%%%%%%%%%%%%%%%

\subsection{Stresses}

In this paper, we will always fix some dimension
$d$ and some graph $G$ with $n$ vertices and $m$ edges.

Let $\p$ be a \defn{configuration} of $n$ points in $\RR^d$.
Let $\omega =( \dots,
 \omega_{ij}, \dots )$ be an assignment of a real scalar
 $\omega_{ij}=\omega_{ji}$ to each edge, $\{i,j\}$ in $G$.  (We have
 $\omega_{ij}=0$, when $\{i,j\}$ is not an edge of $G$.)
We say that
 $\omega$ is an \defn{equilibrium stress vector}
for $(G,\p)$ if the vector equation
\begin{equation}\label{eqn:equilibrium}
 \sum_j \omega_{ij}(\p_i-\p_j) = 0
 \end{equation}
holds for all vertices $i$ of $G$.
The equilibrium stress vectors of $(G,\p)$ form the co-kernel of its
rigidity matrix $R(\p)$.

We associate an $n$-by-$n$
\defn{equilibrium stress matrix} $\Omega$
of $(G,\p)$
to a stress vector $\omega$ of $(G,\p)$,
by setting the $i,j$th entry of $\Omega$ to
$-\omega_{ij}$, for $i \ne j$, and the diagonal entries of $\Omega$
are set such that the row and column sums of $\Omega$ are zero.

For each of the  $d$ spatial dimensions,
if we define a vector $\v$ in $\RR^n$ by collecting  the  associated
coordinate over all of the points in $\p$, we have $\Omega \v=0$.
Thus if the
dimension of the affine span of the vertices $\p$ is $d$, then the
rank of $\Omega$ is at most $n-d-1$, but it could be less.

We define the equilibrium stress matrices
of $G$ in dimension $d$ to be the union of the sets of 
equilibrium
stress matrices for $(G,\p)$
over all $\p$ in $\RR^d$. These are the
symmetric $n$-by-$n$ matrices of rank
$n-d-1$ or less, with the all ones vector in
its kernel, and zero entries corresponding to non-edge pairs of distinct vertices.

In what follows we will drop the word
``equilibrium'' as well as the dimension
for brevity, and just call these stress matrices for $G$.

\subsection{Kernel frameworks}
We will also go in the reverse direction.
We will start with a stress matrix and find its associated
frameworks.
To this
end, we introduce homogeneous coordinates for a
configuration $\p$ of $n$ points in $\RR^d$ by
$\hat \p_i = (\p_i ; 1)$; i.e., appending a $1$
in the last dimension.  Denote by $\hat \p$ the resulting vector
configuration in $\RR^{d+1}$ .  We note
for later that $\p$ is in \defn{affine general
position} (no subset of $k \le d+1$ points of $\RR^d$
on a $(k-2)$-dimensional affine space)
if and only if $\hat \p$ is in
\defn{linear general position}
(no subset of $k \le d+1$ vectors of $\RR^{d+1}$ in a $(k-1)$-dimensional linear space).

Now, fixing a stress matrix $\Omega$ for $G$,
we say that $(G,\p)$ is a \defn{kernel framework}
for $\Omega$ if $\Omega P =0$, where $P$
has as its rows the vectors $\hat\p^t_i$.
If $\p$ has $r$-dimensional affine span and
is a kernel framework for $\Omega$,
then $\Omega$ has rank at most $n - r - 1$,
but it might be lower.

\subsection{Rigidity}
The following definitions are standard.

Two frameworks $(G,\p)$ and
$(G,\q)$ are \defn{equivalent} if
\[
    \| \p_j - \p_i\| = \| \q_j -  \q_i\|
    \qquad \text{(all edges $\{i,j\}$ of $G$)}
\]
They are \defn{congruent} if there is a Euclidean motion
$T$ of $\RR^d$ so that
\[
    \q_i = T(\p_i)\qquad \text{(all verts. $i$ of $G$)}
\]
A framework $(G,\p)$ is \defn{rigid} if there is a neighborhood $U\ni \p$
of configurations in $\RR^d$
so that if $\q\in U$ and $(G,\q)$ is equivalent to $(G,\p)$, then $q$ is
congruent to $p$.  A framework $(G,\p)$ is \defn{globally rigid} if \emph{any}
$(G,\q)$ in $\RR^d$
equivalent to $(G,\p)$ is congruent to it.
framework $(G,\p)$ is \defn{universally rigid} if \emph{any}
$(G,\q)$, in any dimension,  equivalent to $(G,\p)$ is congruent to it.

A configuration is \defn{generic} if its
coordinates are algebraically independent over $\QQ$.
\begin{theorem}[\cite{Asimow-Roth-I,Gortler-Thurston}]
Let $d$ be a dimension and $G$ a graph.  Then either every generic
framework $(G,\p)$ in dimension $d$
is (globally) rigid or no generic framework is
(globally) rigid.
\end{theorem}
We summarize this theorem by saying that rigidity \cite{Asimow-Roth-I} and global
rigidity \cite{Gortler-Thurston} are
\defn{generic properties}.
If every generic framework $(G,\p)$ in dimension $d$ is (globally)
rigid, we say that $G$ is \defn{generically (globally) rigid}
in dimension $d$.

\paragraph{Rigidity Certificates}
One important way to certify that
a framework is  rigid is via the stronger property of
infinitesimal rigidity:
The \defn{rigidity matrix} $R(\p)$ of a framework $(G,\p)$ is the
matrix of the linear system
\[
    \iprod{\p_j - \p_i}{\p'_j - \p'_i} = 0
    \qquad \text{(all edges $\{i,j\}$ of $G$)}
\]
where the vector configuration $\p'$ is variable.  The kernel
of $R(\p)$ comprises the \defn{infinitesimal flexes} of $(G,\p)$.
When $G$ is a graph with $n\ge d$ vertices, a $d$-dimensional
framework $(G,\p)$ is called \defn{infinitesimally rigid}
when $R(\p)$ has rank $dn - \binom{d+1}{2}$.  Infinitesimal rigidity
implies rigidity \cite{Asimow-Roth-I}.

One important way to certify that
a framework is universally rigid is via
the still stronger property of super stability.
A framework
with $d$-dimensional affine span
is \defn{super stable} if it has a positive semidefinite (PSD)
equilibrium stress matrix $\Omega$ of rank $n - d - 1$
and its edges directions are not on a conic at infinity
(a technical property discussed next).

The \defn{edge directions} of a framework $(G,\p)$ is the configuration
$\e$
of $|E|$ points at infinity $\e_{ij} := \p_j - \p_i$.
 A framework \defn{has its edge directions on a conic at infinity}
if there is a quadric surface $\QQQ$ at infinity
containing all of $\e$.
Notably, if
a framework is infinitesimally rigid, then its edge
directions cannot be on a conic at infinity.

The main connection between these concepts is due to Connelly.
\begin{theorem}[\cite{Connelly-energy}]\label{thm:sstour}
If $(G,\p)$ is super stable, then it is universally rigid.
\end{theorem}

% The next theorem is from
% Alfakih and Ye~\cite{Alfakih-Ye-general-position},
% which is a strengthening of a result
% from\cite{Connelly-energy}.
% \begin{theorem}
% \label{thm:alfss}
% If $(G,\p)$ is a framework in $\RR^d$
% with a $d$-dimensional affine span and in
% general affine position
% and it
% has a PSD equilibrium stress matrix of rank $n-d-1$, then
% it is super stable and thus universally rigid.
% \end{theorem}

\paragraph{Generic Global Rigidity}

Connelly~\cite{Connelly-global}
proved the following sufficient condition  for
generic global rigidity of a graph.
\begin{theorem}
\label{thm:suff}
If some generic framework $(G,\p)$ in $\RR^d$
has an (even indefinite) equilibrium stress matrix
of rank $n-d-1$, then the graph $G$ is
generically globally rigid in $\RR^d$.
\end{theorem}

Gortler Healy and Thurston~\cite{Gortler-Thurston} proved the following:
\begin{theorem}
\label{thm:necc}
If some generic framework $(G,\p)$ in $\RR^d$
does not have equilibrium stress matrix
of rank $n-d-1$, then the graph $G$ is not
generically globally rigid in $\RR^d$.
% It is, in fact, generically not globally rigid in $\RR^d$.
% Thus, if a graph $G$ is not generically globally rigid in $\RR^d$
% then it is
% generically not globally rigid in $\RR^d$.
\end{theorem}

The following is the easy half of a theorem by
Hendrickson~\cite{Hendrickson-unique}, which we will need below.
\begin{theorem}
\label{thm:hen}
If $G$ is generically globally rigid in $\RR^d$, then it must
be $(d+1)$-connected.
\end{theorem}

\section{PSD general position stresses (mostly review)}

\begin{definition}\label{def:psd gstress}
Fix a dimension $d$.
Let $G$ be a graph with $n \ge d+2$
vertices.
Let \defn{$\PGST$} be the real semi-algebraic set of $n$-by-$n$
\defn{positive semidefinite
general position $d$-dimensional stress matrices} for $G$.
Specifically, this is the set of real
PSD symmetric matrices that have $0$ entries
corresponding to non-edges of $G$,  have the all-ones vector in its
kernel, have  rank equal to $n-d-1$, and such
that every subset of $n-d-1$ columns is a linearly independent set.
\end{definition}

The main theorem of this section is the following
\begin{theorem}
\label{thm:pgstr}
Fixing a dimension $d$,
let $G$ be $(d+1)$-connected. Then the set of
positive semidefinite general position
(real)
stress matrices for $G$
is an irreducible real
semi-algebraic set of dimension
$m-\binom{d+1}{2}$.

If $G$ is not $(d+1)$-connected then
this set is empty.
\end{theorem}

This result is mostly just a repackaging of
a result from~\cite{ghave}, which uses a construction of
Alfakih's~\cite{Alfakih-conn}, that is, in turn,  based on the
work of~\cite{Lovasz-Schrijver}. Let's work through this backwards.

Before going on,
we recall that Gale duality is the following  linear algebra statement.
\begin{lemma}\label{lem:linear-gale}
Let $A$ be an $m\times n$ matrix of rank $n-r$ with entries in any
field and $X$ an $n\times r$ matrix with columns that span the kernel
of $A$.  Then a subset of $r$ rows of $X$ is linearly independent if and
only if the complementary set of $n-r$ columns of $A$ is linearly independent.
\end{lemma}
See \cite{sbs} for a nice proof.

Let us explicitly restate a
general position result used in~\cite{Alfakih-conn} to relate
general position in stresses to general position kernel frameworks.
\begin{lemma}\label{lem:pdual}
Let $G$ be $(d+1)$-connected.  If $\Omega\in \PGST$ for  $G$,
and $\p$ affinely spans $\RR^d$ and is in the kernel of $\Omega$,
then $\p$ is in affine general position.  If $\p$ is in affine general position,
then any rank $n - d - 1$
PSD stress matrix $\Omega$ with $\p$ in its kernel is
in $\PGST$.
\end{lemma}
\begin{proof}
This follows from Gale  duality (see Lemma
\ref{lem:linear-gale}).  We give the details
for completeness.

Let $\p$ be an affinely spanning configuration
of $n$ points in $\RR^d$ and $\Omega$ a
stress matrix of rank $n - d - 1$ with $\p$ in its kernel.
By general position, the vector configuration $\hat \p$ is
in linearly general position.  If the matrix $P$
has the vectors $\hat \p^t_i$ as its rows then
$\Omega P = 0$, and so the columns of $X$ span
the kernel of $\Omega$.  Gale duality then implies
that any $n - d - 1$ columns $\Omega$ are linearly
independent.  Hence $\Omega$ is a general position stress.

Going in the other direction, suppose that $\Omega$
is a general position stress matrix.  Let $P$ be any
matrix with columns spanning the kernel of $\Omega$
where the last column is all ones.  This is possible
because $\Omega$ is a stress matrix.  Gale duality
now implies that the rows of $P$ are in linearly
general position.  Interpreted as homogeneous coordinates,
we see that $P$ corresponds to a configuration $\p$ in
affine general position.  Since every affinely spanning
kernel framework arises this way they are all in general position.
\end{proof}

\subsection{GORs and connectivity}

In~\cite{Lovasz-Schrijver}, Lov\'asz
Saks and Schriver define a concept called a
(GOR) general position
orthogonal representation of a graph $G$ in $\RR^D$.

\begin{definition}
\label{def:GOR}
Let $G$ be a graph and
let $D \ge 1$ be a fixed dimension.
An (OR)  \defn{orthogonal representation} of $G$ in $\RR^D$
is a vector configuration $\v$ indexed by the vertices of $G$ in
$\RR^{D}$ with the following  property:
$\v_i$ is orthogonal to the vectors
associated with each non-neighbor of vertex $i$.
The set of ORs is an algebraic set.
%(defined over $\QQ$).

A (GOR) \defn{general position orthogonal representation} of $G$
in $\RR^D$
is an OR in $\RR^D$ with the added property that
the $\v_i$ are in general linear position.
The set of GORs is a quasi-projective variety.
%(defined over $\QQ$).

\end{definition}

The relevant results from \cite{Lovasz-Schrijver} are the following.
\begin{theorem}\label{thm:lss}
Let $n>D$.
Let $G$, a graph on $n$ vertices,
be  $(n-D)$-connected for some $D$.  Then $G$ must
have a GOR in $\R^D$ \cite[Theorem 1.1]{Lovasz-Schrijver}.
Moreover, the set of all such GORs of $G$ is an irreducible
quasi-projective variety~\cite[Theorem 2.1]{Lovasz-Schrijver}.
If $G$ is not $(n-D)$-connected, then it cannot
have a GOR in $\R^D$.
\end{theorem}

In our terminology, we will set $D:=n-d-1$ where $d$ is fixed, and will consider graphs that are $(d+1)$-connected.  These graphs
have at least $d+2$ vertices, so $D \ge 1$.
With this notation, the theorem tells us that
 we need
$(d+1)$-connectivity to obtain GORs in $\RR^{n-d-1}$.

\begin{definition}
Let $G$ be a $(d+1)$-connected graph with $n$ vertices,
for some $d$.
Denote by $D_G$ the dimension of the set its GORs in $\RR^{n-d-1}$.
\end{definition}

The following is from~\cite{ghave}.
\begin{corollary}\label{cor:gor-dim}
Let $G$ be a $(d+1)$-connected graph with $n$ vertices and
$m$ edges.  Then the dimension $D_G$  is $n(n-d)-\binom{n+1}{2}+m$.
\end{corollary}

\subsection{LSS stresses}
Because of the orthogonality property of a GOR, its
Gram matrix has the right zero/non-zero pattern
to be a stress matrix.  Alfakih \cite{Alfakih-conn}
builds on this.  First we set some notation.
\begin{definition}\label{defn:gorc}
Let $G$ be a $(d+1)$-connected
graph and $\v$ a GOR of $G$ in dimension
$n - d -1$.  Since $G$ is $(d+1)$-connected,
$n\ge d + 2$, so $n - d - 1 > 0$.

The $(n - d - 1)\times n$ matrix $X$
with the $\v_i$ as its columns is the \defn{configuration
matrix} of $\v$.  We denote the Gram matrix $X^tX$ of
$\v$ by $\Psi$.  Note that $\Psi$ is, by construction, PSD
and  has
rank $n - d - 1$, (as $\v$ is in general position).

A GOR $\v$ is called \defn{centered} if its
barycenter is the origin.  We define
$\GOR^0$ to be the set of
centered GORs.

The Gram matrix $\Psi$
is a stress matrix (which we will call $\Omega$)
if and only if $\v$ is centered.  (Recall that the
extra condition is that the all-ones vector is in the
kernel.) Such an $\Omega$ is  PSD and of
rank $n-d-1$.

We define the set $\LL$ of \defn{Lovász-Saks-Schrijver stresses}
to be the collection of stress matrices $\Omega$ arising
as the Gram matrices of centered GORs.  Denote its
dimension by $D_L$.
\end{definition}

A \defn{full rank centering map} of a GOR $\v$ is a set of non-zero scalars
$\alpha_i$ so that the configuration
\[
    \{\alpha_1 \v_1, \ldots, \alpha_n \v_n\}
\]
has its barycenter at the origin.
Alfakih's main result in \cite{Alfakih-conn} is the following.
\begin{theorem}[\cite{Alfakih-conn}]
\label{thm:alf}
Let $G$ be $(d+1)$-connected.  Then any GOR
in $\RR^{n-d-1}$ for $G$ has a full rank centering map.
This gives rise to stress matrix $\Omega$ in LSS.  Moreover,
any framework $(G,\p)$ with $d$-dimensional affine span, that has
$\Omega$ as an equilibrium stress matrix, must
be in affine general position and be super stable.
\end{theorem}

The following is from~\cite{ghave}.
\begin{corollary}
\label{cor:dl}
If $G$ is  $(d+1)$-connected then $D_L=m-\binom{d+1}{2}$
and LSS is an irreducible semi-algebraic set.
If $G$ is not $(d+1)$-connected, then LSS is empty.
\end{corollary}

\subsection{Finishing the Proof}
We are nearly ready for the proof of Theorem 
\ref{thm:pgstr}.

\begin{lemma}
\label{lem:pgeq}
  $\PGST$ is equal to
 LSS.
\end{lemma}
\begin{proof}
Every matrix in LSS has all of these properties.
In particular, the general position of the
vector configurations in
$GOR^0$ gives rise the the required linear independence
of column subsets.

Going the other way, every matrix $\Omega$ in $\PGST$ can be
shown to be in LSS.
By assumption, $\Omega$ is of rank $n-d-1>0$.
It can be
Takagi factored, $\Omega= X^tX$, with the columns of $X$ describing  a real
vector
configuration in $\RR^{n-d-1}$.  This configuration must have
all of the properties of
a centered OR. The linear independence of
each set of $n-d-1$ columns from $\Omega$
places the columns of $X$ in general position.
\end{proof}

\begin{proof}[Proof of Theorem~\ref{thm:pgstr}]
We simply combine
Corollary~\ref{cor:dl} with
Lemma~\ref{lem:pgeq}.
\end{proof}

\section{General position stresses}
With the necessary background in place, we define our main 
new concept.

\begin{definition}\label{def:gstress}
Fix a dimension $d$.
Let $G$ be a graph with $n \ge d+2$
vertices.
Let \defn{$\ST$} be the real quasi projective variety of
real symmetric
$n$-by-$n$ matrices that have $0$ entries
corresponding to non-edges of $G$,  have the all-ones vector in its
kernel,
and have  rank equal to $n-d-1$.

Let \defn{$\GST$} be the real quasi projective variety of $n$-by-$n$
\defn{general position $d$-dimensional stress matrices} for $G$.
Specifically, this is the
matrices in $\ST$
such
that every subset of $n-d-1$ columns is linearly independent.
We call such a matrix a \defn{Gstress}.
\end{definition}

The following is the central result of this paper.

\begin{theorem}
\label{thm:gstr}
Fixing a dimension $d$,
let $G$ be $(d+1)$-connected. Then the set of general position
stress matrices for $G$
is an irreducible real quasi-projective variety of dimension
$m-\binom{d+1}{2}$.

If $G$ is not $(d+1)$-connected,
then this set is empty. \end{theorem}

Again, we have a lemma
that allows us to transfer general position between stresses
and their kernel frameworks.
\begin{lemma}\label{lem:dual}
Let $G$ be $(d+1)$-connected.  If $\Omega\in \GST$ for  $G$,
and $\p$ affinely spans $\RR^d$ and is in the kernel of $\Omega$,
then $\p$ is in affine
general position.  If $\p$ is in
affine general position,
then any rank $n - d - 1$
stress matrix $\Omega$ with $\p$ in its kernel is
in $\GST$.
\end{lemma}
\begin{proof}
The proof of Lemma \ref{lem:pdual} did not use the
signature of $\Omega$ so it works verbatim here.
\end{proof}

The plan is  to first study
GORs in the complex setting, then study
LSS stress matrices in the complex setting. 
We will obtain results for the complex setting that
are, essentially, the same as those from the real PSD 
one.
Finally we take
the real locus of this set of complex stress matrices, 
which will be $\GST$.

\subsection{Complex GORs}
\begin{definition}
\label{def:CGOR}
Let $G$ be a graph and
let $D \ge 1$ be a fixed dimension.
An (COR)  \defn{complex
orthogonal representation} of $G$ in $\CC^D$
is a vector configuration $\v$ indexed by the vertices of $G$ in
$\CC^{D}$ with the following  property:
if $\v_i = (\alpha_1, \ldots, \alpha_D)$
and $i$ is not a neighbor of
$j$ and $\v_j = (\beta_1, \ldots, \beta_D)$,
then
\[
     \sum_{k=1}^D \alpha_k\beta_k = 0
\]
Notice that this ``algebraic dot product'' is defined without conjugation.
The set of CORs form a complex algebraic set.
%(defined over $\QQ$).

A (CGOR) \defn{general position complex orthogonal representation} of $G$
in $\CC^D$
is a COR in $\CC^D$ with the added property that
the $\v_i$ are in general linear position.
The set of CGORs form a complex
quasi-projective variety
set.
%(defined over $\QQ$).
\end{definition}

\begin{theorem}\label{thm:cgor-dim}
Let $n>D$.
Let $G$ be a $(n-D)$-connected graph with $n$ vertices.
Then the set of CGORs in $\CC^D$ is non-empty and is an irreducible
quasi projective variety.
If $G$ is not $(n-D)$-connected, then it cannot
have a CGOR in $\CC^{D}$, or
%even
a COR in
which the non-neighbors of every vertex are
represented by linearly independent vectors.
\end{theorem}

\begin{proof}
First we suppose that $G$ is $(n-D)$-connected.
From Theorem~\ref{thm:lss}, we have the
existence of a GOR, which is also a CGOR.
Thus the set of CGOR  is non-empty.
The proof of the irreducibility of the GORs
from~\cite{Lovasz-Schrijver} carries over
identically to the complex setting, thus the
set of CGORs is irreducible as well.
This gives us the first half of theorem.

Because orthogonality under our bilinear form
does not imply linear independence in the complex setting, we need to
deal with the low-connected case slightly differently
than \cite[Theorem 1.1$'$]{Lovasz-Schrijver}.
Suppose that $G$ is not $(n-D)$-connected.
  The graph $G$ is the union of vertex-induced
subgraphs $G_1 = (V_1, E_1)$ and $G_2 = (V_2, E_2)$ so that
$V_1\cap V_2$ has at most $n-D-1$ vertices.  For convenience,
set $X_1 = V_1\setminus V_2$ and $X_2 = V_2\setminus V_1$.  Both
of these sets are non-empty. We also have $|X_1| + |X_2| \ge D+1$.

Now we consider a COR $\v$ for $G$ in $\CC^{D}$.
Suppose first that the vectors $\{\v_i : i\in X_1\}$ have
a $D$-dimensional span.
From the bilinearity of the complex algebraic dot product,
for each $j\in X_2$, we have the linear constraints
\[
    \v_j\cdot \v_i = 0 \qquad \text{(all $i\in X_1$)}
\]
The $|X_2|$ vectors
of $\v$ corresponding to vertices in $X_2$ are, therefore
all zero vectors.
We conclude that $\v$
is not in general position.  Moreover, the non-neighbors of any vertex
in $X_1$ contain all of $X_2$, so there is a vertex with its non-neighbors
linearly dependent.

The same is true if the $X_2$ vectors have a $D$-dimensional span.
So going forward, let us assume that neither  the $X_1$ or $X_2$
vectors have a $D$-dimensional span.

Suppose that there was a linear dependency in either the $X_1$
or $X_2$ vectors. Since neither set has  full span, this would
rule out general position. Moreover, the non-neighbors of any vertex
in $X_i$ contain all of $X_j$, so there is a vertex with its non-neighbors
linearly dependent.

So going forward, let us assume that both the $X_1$ and the $X_2$
vectors are linearly independent.

The returning to the constraints,
for each $j\in X_2$, the linear constraints
\[
    \v_j\cdot \v_i = 0 \qquad \text{(all $i\in X_1$)}
\]
are linearly independent.  The $|X_2|$ vectors
are constrained
to lie in a subspace of dimension at most
\[
    D - |X_1|
\]
and by linear independence, have a cardinality with this same bound.
But this would contradict $|X_1| + |X_2| \ge D+1$.
\end{proof}

In our terminology, we will set $D:=n-d-1$ where $d$ is fixed, and will consider graphs that are $(d+1)$-connected.  These graphs
have at least $d+2$ vertices, so $D \ge 1$.
With this notation, the theorem tells us that
 we need
$(d+1)$-connectivity to obtain
CGORs in $\CC^{n-d-1}$.

\begin{definition}
Let $G$ be a $(d+1)$-connected graph with $n$ vertices,
for some $d$.
Denote by $D_{CG}$ the complex dimension of the set its CGORs in $\CC^{n-d-1}$.
\end{definition}

\begin{theorem}\label{thm:cgor-dim2}
Let $G$ be a $(d+1)$-connected graph with $n$ vertices and  $m$ edges.  Then the set of CGORs has  dimension, $D_{CG}$, equals $D_G$.
\end{theorem}
The proof for the dimension count,
is then no different from that
of Corollary~\ref{cor:gor-dim}
(which is done in~\cite{ghave}).

\subsection{Complex LSS}
We now apply the LSS construction to the complex
setting.

\begin{definition}\label{defn:cgorc}
Let $G$ be a $(d+1)$-connected
graph and $\v$ a CGOR of $G$ in dimension
$n - d -1$.  The $(n - d - 1)\times n$ matrix $X$
with the $\v_i$ as its columns is the \defn{configuration
matrix} of $\v$.  We denote the
algebraic-Gram
matrix $X^tX$ (with no conjugation)
of
$\v$ by $\Psi$.  Note that $\Psi$ has
rank $n - d - 1$, (as $\v$ is in general position
and using Lemma~\ref{lem:sqrank}).

A CGOR $\v$ is called \defn{centered} if its
barycenter is the origin.  We define
$CGOR^0$ to be the complex quasi-projective variety
of centered CGORs.

The Gram matrix $\Psi$
is a (complex)
stress matrix (which we will call $\Omega$)
if and only if $\v$ is centered.  (Recall that the
extra condition is that the all-ones vector is in the
kernel.)

We define the set $CLSS$
to be the collection of stress matrices $\Omega$ arising
as the Gram matrices of centered CGORs.  Denote its complex
dimension by $D_{CL}$.
\end{definition}

\begin{theorem}
\label{thm:cdl}
If $G$ is  $(d+1)$-connected, then
$D_{CL}=D_L$ and CLSS is irreducible.

If $G$ is not $(d+1)$-connected then CLSS is
empty.
\end{theorem}
\begin{proof}
This follows using the proofs of
Theorem~\ref{thm:alf} from~\cite{Alfakih-conn}
and Corollary~\ref{cor:dl} from~\cite{ghave}
work verbatim in the complex setting.
\end{proof}

\subsection{Finishing the proof}
We will obtain $\GST$ using a Gram matrix construction.  There 
is a technical issue, which is that, for complex matrices 
$A$, we do not necessarily have 
$\Rank(A^t A)=\Rank(A)$.  However, if $A$ has linearly independent
rows, this does hold.
\begin{lemma}
\label{lem:sqrank}
Let $A$ be a complex $D$-by-$n$ matrix
with rank $D$. Then $A^t A$ has rank $D$.
\end{lemma}
This follows from the fact that  $A^t$ must have linearly independent columns by assumption.
Meanwhile, left multiplication by a matrix with 
linearly independent columns does not change the rank.

\begin{lemma}
\label{lem:gstr1}
 $\GST$ is equal to the real locus of CLSS.
\end{lemma}
\begin{proof}
Let $X$ be the real locus of CLSS and let $\Omega\in X$ be given.
By definition, there is a $\v \in \CGOR^0$ so that $\Omega$ is the 
Gram matrix of $\v$.  The configuration matrix of $\v$ has 
$n-d-1$ rows and $n$ columns in linearly general position, 
so Lemma \ref{lem:sqrank} implies that $\Omega$ has 
rank $n-d-1$.  General posiiton of $\v$ also implies that 
the columns of $\Omega$ are in linearly general position, 
so $\Omega$ is a Gstress.  Hence $X\subseteq \GST$.

Now let $\Omega$ be a general position stress.
By assumption, $\Omega$ is of rank $n-d-1>0$.
It can be Takagi factored, $\Omega= X^tX$, with the columns of 
$X$ describing a  vector configuration in $\CC^{n-d-1}$.
Because $\Omega$ is a stress matrix, $X$ will correspond to 
a centered orthogonal representation of $G$.  

To conclude,
we need to check that this orthogonal representation is 
in general position.  Suppose the contrary.  Then we have 
some $k\le n - d - 1$ columns of $X$ that are linearly
dependent.  Let $X'$ denote the submatrix of $X$ 
obtained by throwing away the other columns.  Then 
$X^t X'$ corresponds to $k$ columns of $\Omega$ that 
are linearly dependent, contradicting that $\Omega$ 
is a general position stress.  Hence $X$ is in general 
position, and, thus, defines an element of $\CGOR^0$. 
\end{proof}
Now we compute the dimension of the GStresses.
\begin{lemma}
\label{lem:gstr2}
 If $G$ is $(d+1)$-connected, then $\GST$ has dimension equal to  $D_L$
\end{lemma}
\begin{proof}
$\GST$ includes $\PGST$ and so must be of real dimension
at least $D_L$.

On the other hand the as the real locus of a complex
quasi-projective variety (Lemma~\ref{lem:gstr1}),
its real dimension cannot
be larger than the complex dimension $
D_{CL}$, which
is equal to $D_L$~\cite{whitney}.

\end{proof}

\begin{lemma}
\label{lem:gstr3}
If $G$ is $(d+1)$-connected, then
  $\GST$ is an irreducible real quasi-projective variety.
\end{lemma}
\begin{proof}
The lemma follows from Lemma~\ref{lem:locus} in light of
lemmas~\ref{lem:gstr1} and~\ref{lem:gstr2} and Theorem~\ref{thm:cdl}.
\end{proof}

\begin{proof}[Proof of Theorem~\ref{thm:gstr}]
This combines
Lemmas~\ref{lem:gstr2} and
\ref{lem:gstr3}.
%and~\ref{lem:lowcon}.
\end{proof}

\subsection{An easier special case}
We have used CGORs to describe the general position
stresses of any $(d+1)$-connected graph $G$ with $m$ edges.  In
the special case that $G$ also has
$K_{d+1}$ as a subgraph, we can obtain the same characterization
of its general position stresses without resorting to GOR
technology.
In particular,
we can parameterize Gstr by a
Zariski open subset of $\RR^{m - \binom{d+1}{2}}$ using a
more elementary
result of Linial, Lovász and Widgerson \cite{llw}.
\begin{lemma}
\label{lem:param1}
Let $G$ be a $(d+1)$-connected graph with $m$ edges  and suppose that $H$ is a
$K_{d+1}$ subgraph of of $G$.  Then, for a Zariski open subset of assignments
$\w_{G\setminus H}$
of weights to the of edges of $G$ not contained in $H$, there is a
unique general position stress $\Omega$ that agrees with $\w_{G\setminus H}$
on the edges of $G\setminus H$.
\end{lemma}
\begin{proof}
It is well-known (e.g., \cite{llw}), that for
$\w_{G\setminus H}$
consisting of positive real numbers and fixed,
and affinely independent positions $\p_H$ for the vertices of $H$,
there
is a unique, configuration $\p$ in dimension $d$
that agrees with $\p_H$ on the vertices of $H$ with the additional
property that the vertices of $G\setminus H$ are in equilibrium with the
weights $\w_{G\setminus H}$.  This $\p$ can be found as the unique solution to  a non-singular
linear system depending on $G$, $\p_H$, and $\w_{G\setminus H}$%
\footnote{This system arises as the gradient of a convex energy function,
but we don't need that here.  All that is important is that it can be solved.}
Since matrix singularity is an algebraic condition,
this existence and uniqueness is true
over a Zariski open subset
even when allowing
negative numbers.

Once we have $\p_{G\setminus H}$, we can use it to complete
$\w_{G\setminus H}$ to an equilibrium stress as follows.
The resultant forces%
\footnote{Computing these is
where we need to known the vertex positions.}%
on $H$
are an equilibrium load by Lemma \ref{lem: eq load support}.
Since $(H,\p_H)$ is statically rigid and independent, this load has a
unique resolution $\w_H$ by Lemma \ref{lem: static inf}.
Combining $\w_H$ and
$\w_{G\setminus H}$, we get equilibrium at every vertex. So this process
gives us
a unique $\Omega$ for
fixed $\w_{G\setminus H}$
and  $\p_H$.

As a stress for $\p$, $\Omega$ must have
rank at most $n-d-1$.
Given $\p_H$,
the rest of
$\p$ was uniquely determined by
the equillibrium condition
of our linear system,
so $\Omega$ must have rank
equal to
$n-d-1$.
Moreover its kernel frameworks must be
the affine transforms of $\p$.

Notice that in our linear system,
if we change
the positions of $\p_H$,
a process that
can be modeled as an affine transform,
then the solution $\p$ will undergo
the same affine transform.
$\Omega$ will  also a be a
stress for this
affinely transformed $\p$, and so must
be the unique stress completion
that will be obtained
from our construction.
Thus $\w_{G\setminus H}$ has a
unique stress completion.

Finally we need to establish
the general position property.
It is shown in
\cite[Theorem 2.4]{llw}, that for
$\w_{G\setminus H}$ consisting of
generic positive real numbers and fixed,
affinely independent positions $\p_H$ for the vertices of $H$,
that the resulting $\p$ is
in affine general position in dimension $d$.
This result also holds
generically
when allowing negative numbers, and
so holds over a Zariski open subset.
Since $\p$ is in affine
general position Lemma \ref{lem:dual}
says the stress we obtained is a general position stress.

\end{proof}

\begin{definition}
Let $G$ be a $(d+1)$-connected graph with a $K_{d+1}$ subgraph $H$.  We define
the map $\rho : X \to \ST$ to be the one from Lemma \ref{lem:param1}
where
%the vertices of $H$ are on the standard simplex and
$X\subseteq \RR^{m - \binom{d+1}{2}}$ is the
implicit Zariski open domain where the linear system is non-singular.
\end{definition}

The map $\rho$ is clearly injective.
Now we argue that is it surjective onto $\GST$.
\begin{lemma}
\label{lem:param2}
Let $G$ be a $(d+1)$-connected graph with $m$ edges  and suppose that $H$ is a
$K_{d+1}$ subgraph of of $G$.
Let $\Omega$ be a general position stress.
Let $\w_{G\setminus H}$ be its values on the
edges not contained in $H$. Then $\rho(\w_{G\setminus H})$ is well defined
and equal to $\Omega$.
\end{lemma}
\begin{proof}
By our Gstr assumption, any affinely spanning $\p$ in the kernel of
$\Omega$
(which must be in general position from Lemma~\ref{lem:dual})
is a non-singular affine image
of a unique
kernel configuration $\p_0$ with the vertices of $H$ the standard simplex.
Meanwhile, due to the rank of $\Omega$, there is only one $\p$ in the kernel, up to an affine transform.
Thus the linear system used to define the map $\rho$ is non-singular,
so the map is well defined here.
Hence we can use the construction of Lemma~\ref{lem:param1}
to uniquely complete the weights on the edges of $G\setminus H$ coming from $\Omega$.
This unique completion  must be $\Omega$,
since $\Omega$ is a completion.
\end{proof}

We summarize this section by the following.
\begin{theorem}
Let $G$ be a $(d+1)$-connected graph with $m$ edges  and suppose that $H$ is a
$K_{d+1}$ subgraph of of $G$.
Then $\GST$ is irreducible, of dimension $m - \binom{d+1}{2}$, and parameterized by
a Zariski open subset of $\RR^{m - \binom{d+1}{2}}$,
representing the weights on the edges of $G\setminus H$,
under a rational map.
\end{theorem}
% \begin{proof}
% Lemmas \ref{lem:param1} and \ref{lem:param2} say that the rational map $\rho$ is a
% bijection between a Zariski open subset of $\RR^{m - \binom{d+1}{2}}$.
% \end{proof}
This result agrees with Theorem \ref{thm:gstr}, while circumventing the machinery
or orthogonal representations.  The proof also shows that the dimension
$m - \binom{d+1}{2}$ appearing in Theorem \ref{thm:gstr} has a natural
interpretation when there is a $K_{d+1}$ subgraph: we can freely assign weights to
the rest of the edges and then compensate on edges of the $K_{d+1}$.  That this
dimension stays the same without a $K_{d+1}$ subgraph is interesting
and somewhat deeper.

\section{General position stressable frameworks}
\label{sec:gstressable}

Now we turn the tables and look at frameworks that are
in equilibrium under a Gstress.

\begin{definition}

  Let $d$ be a dimension and
  let $G$ be a graph with $n \ge d+2$
vertices.
 We define
  the \defn{Gstressable frameworks} to be the set of
  frameworks $(G,\p)$ in $\RR^d$
  %with   $d$-dimensional affine spans,
  such that $\p$
  is in general affine position, and has
  an equilibrium stress of rank $n-d-1$.
\end{definition}

\begin{lemma}
\label{lem:gsable}
A framework is Gstressable
iff it has full span and a stress from $\GST$.
\end{lemma}
\begin{proof}
See Lemma~\ref{lem:dual}.
\end{proof}

The main result of this section is the following
\begin{proposition}
\label{prop:gstConst}
The Gstressable frameworks form a real
irreducible constructible set.
\end{proposition}
The basic idea is to create a rational map
taking each stress in $\GST$ to a
kernel framework.
To make this map well defined, we mod out by affine transforms.
In particular we  select $d+1$ vertices
and pin them to canonical locations. Then for each stress, we can
set up a linear system which will determine a   kernel
framework.
This linear system will be non-singular, unless the stress $\Omega$
is of rank $< n-d-1$, or when the kernel framework of $\Omega$
has these selected vertices with deficient span.
Neither of these can happen for a Gstress.

Using this approach, we can immediately conclude that
the Gstressable frameworks,
as the image of an irreducible semi-algebraic set under a rational map,
is irreducible
and is semi-algebraic. But with a bit more work, we can upgrade this conclusion
to real-constructible.
\begin{proof}
We will work over the complex general position stresses.
In the complex setting, we know that the image of
an irreducible constructible set under a rational map is
constructible.
In our case,
 the image, which we denote by 
$X$,  consists of complex  frameworks (in pinned position)  with complex general position stresses.

Next we enlarge
our set to include all non-singular affine
images of each framework in $X$.
To do this we build an irreducible  bundle over $X$ where
the points are $(\p,A)$ with $A$ a non-singular affine transformation.
Then we apply the polynomial map $(\p,A)\mapsto (A(\p_i))_{i=1}^n$.  Hence,
the affine orbit of $X$ is constructible and irreducible.

This gives us the irreducible constructible set $S$ of complex frameworks in general position with complex stresses of
rank $n-d-1$.

Now we take the real locus of $S$ the set of real frameworks
in general position with
with complex general stresses of rank $n-d-1$. This must  be a real constructible set.

Finally, we need to show that every framework
in this real locus also has a real
stress of rank $n-d-1$. This is done
in the next lemma, which completes the proof.
\end{proof}

\begin{lemma}
Let $S$ be the set of complex frameworks in general position
with complex stresses of rank $n-d-1$.
Denote by $\operatorname{Re}(S)$ the real locus of $S$.
Then $\operatorname{Re}(S)$  is equal to the 
 set of real frameworks
in general position with a real stress of rank
$n-d-1$.
\end{lemma}
\begin{proof}
One direction is trivial: by definition every Gstressable framework has a 
stress of rank $n-d-1$ and so is in $\operatorname{Re}(S)$.

For the other direction, every framework in $\operatorname{Re}(S)$ is in
general position, so we just need to show a real stress of
rank $n-d-1$. Fixing a $\p \in \operatorname{Re}(S)$, its complex stress space
is just the co-kernel of its rigidity matrix. Since the
rigidity matrix is real valued, this complex space has a
basis of real vectors. This makes the complex stress space
equal to the complexification of the real stress space.
This, in turn, implies that a generic stress from the real
stress space is also generic in the complex stress space.
Since every generic stress in the complex stress space has rank
$n-d-1$, we are done.
\end{proof}

\section{Locally Spanning}
General position is a strong restriction on a framework or a GOR.  In 
this section, we relax general position to a condition on vertex 
neighborhoods that is enough to get results similar to those in 
Section \ref{sec:gstressable}.

\begin{definition}
Let $G$ be a graph with $n \ge d+2$
vertices.
  We define
  the \defn{Fstressable} frameworks to be the set of
  $d$-dimensional frameworks $(G,\p)$ such that
  each vertex neighborhood% 
  \footnote{This includes the points corresponding to a vertex and its 
  neighbors.}
  in $\p$ has a full $d$-dimensional affine span and such that
  $\p$ has
  an equilibrium stress of rank $n-d-1$.
\end{definition}

Our main results in this section will be the
following, which say, informally, that the Fstressable 
frameworks can be approximated by Gstressable ones.

\begin{theorem}
\label{thm:fapprox2}
The set of Gstressable frameworks is standard-topology dense in the 
set of Fstressable frameworks.
\end{theorem}

\begin{theorem}
\label{thm:fapprox3}
Fstressable is a contructible set. It is irreducible and has the 
same dimension as Gstressable.
\end{theorem}
To this end we start with the next definition. It may seem a bit
unnatural, but it will prove to be exactly what we need when we dualize
and look at frameworks.

\begin{definition}\label{def: lgor}
  A  \defn{local full spanning orthogonal representation} (for short, a FOR)
  $\v$ of $G$
in $\RR^D$
is an OR in $\RR^D$ with the added following property:
For each vertex $i\in V(G)$, there are neighbors $i_1, \ldots, i_{n-D-1}$
of $i$ so that the set of vectors
\[
    \{\v_j : j\in \{1, \ldots, n\} \setminus \{i, i_1, \ldots, i_{n-D-1}\}\}
\]
are linearly independent.

Any FOR necessarily spans $\RR^D$.
\end{definition}
We point out for later that any FOR has the property that the non-neighbors
of any vertex are represented by linearly independent vectors.

\begin{lemma}
\label{lem:gorfor}
Every GOR is a FOR.
\end{lemma}
\begin{proof}
In a GOR, the configuration  $\v$
is in general position. Thus any
$D$ of the $\v_j$ will be linearly
independent, ensuring we can
satisfy the FOR condition.
\end{proof}

The definition of an FOR relaxes that of a GOR, but not
so
much that Theorem \ref{thm:alf} becomes false.  Here is the
first part of it for FORs.
\begin{lemma}\label{lem: lgor center}
Any  FOR has a full rank centering map.
\end{lemma}
\begin{proof}
The main step of Alfakih's proof that every GOR has a
full rank centering map in \cite{Alfakih-conn} is that
in a GOR, for each $\v_i$, there is a linear dependence among
the vectors in $\v$ with a non-zero coefficient on $\v_i$.
This holds for a FOR because, from the definition,
associated with each $\v_i$ is a set of $D$ vectors
$\v_j$  which span the $D$-dimensional space.
The rest of the proof in \cite{Alfakih-conn} goes through unmodified.
\end{proof}

If we require only that the non-neighbors of every vertex are linearly
independent, then the conclusion of Lemma \ref{lem: lgor center} becomes
false, and also, the span of such an OR may be less than $D$-dimensional.
We need to rule out both of these cases for our intended application.

We also want to generalize the signature of our metric.
\begin{definition}
  We pick of a set of $D$ signs $s_i=\pm 1$. The choice of signs gives us a  symmetric
  bilinear form
 \ba
 \la x, y \ra := \sum_{i=1}^D s_i x_i y_i
 \ea
 With the signature fixed, we say that we are in a \defn{pseudo-Euclidean} setting.
 In this setting, we call two vectors $x$ and $y$ orthogonal if
 $\la x,y\ra = 0$.

 In any pseudo-Euclidean setting we can define the same notion of an OR, GOR, and FOR.
\end{definition}

Results and proofs from \cite{Lovasz-Schrijver} imply that the set
of FORs has nice algebraic and geometric properties.
\begin{lemma}
\label{lem:lgor}
In each of the the real, or pseudo-Euclidean settings,
the set of FORs is irreducible and the set of GORs is
(standard topology) dense in the FORs. If $G$ is
not $(n-D)$-connected, then the set of FORs is empty.
\end{lemma}
\begin{proof}
In the real setting, it is  explicitly stated in ~\cite[Remark
on page 448]{Lovasz-Schrijver} that the GORs are dense
in the class of ORs in which the non-neighbors every
vertex are represented by linearly independent vectors.
This follows from the proof of~\cite[Theorem 2.1]{Lovasz-Schrijver}.
We can apply the same proof for any pseudo-Euclidean
setting by an appropriate sprinkling of negative signs.

The FORs are a subset of this larger class.
The GORs are a subset of FORs from Lemma~\ref{lem:gorfor},
and hence
the
GORs are dense in FOR as well.

For the less connected case, Theorem~\ref{thm:cgor-dim} shows that
$G$ cannot have a FOR in dimension $D$, since a FOR implies
that the non-neighbors of every vertex are represented by linearly
independent vectors.
\end{proof}

In our terminology, we will set $D:=n-d-1$ where $d$ is fixed, and will consider 
graphs that are $(d+1)$-connected.  These graphs
have at least $d+2$ vertices, so $D \ge 1$.
With this notation, the theorem tells us that
 we need
$(d+1)$-connectivity to obtain FORs in dimension 
${n-d-1}$.

At this point, we need to describe how to obtain a
stress matrix from an FOR using Alfakih's construction
from \cite{Alfakih-conn}.
Given a centered FOR $\v$ in some real or pseudo-Euclidean setting
of $G$ in dimension
$n - d -1$, the
$(n - d - 1)\times n$ matrix $X$
with the $\v_i$ as its columns is the \defn{configuration
matrix} of $\v$.  Let $\Psi$ be the
matrix $X^tSX$, where $S$ is the diagonal matrix
representing the pseudo-Euclidean setting we are in.
Then $\Psi$ will have rank $n-d-1$, and its number
of positive and negative eigenvalues will be determined
their quantity in $S$. By the orthogonality
property of an FOR,
$\Psi$ will have
zeros corresponding to the non-edges of $G$. Because $\v$ is centered,
$\Psi$ is a stress matrix.

\begin{definition}
The set FLSS is the set of $n\times n$ stress matrices for a
graph $G$ that arise from applying the above construction to
a centered FOR in the real or some pseudo-Euclidean setting,
for a graph $G$ in dimension $d$.
\end{definition}
We are ready to define the analogue to $\GST$ in the FOR 
setting.
\begin{definition}
Let $G$ be a graph with $n \ge d+2$
vertices.
Let \defn{FST} be the real
quasi-projective variety of $n$-by-$n$
\defn{locally full spanning  $d$-dimensional stress matrices} for $G$.
Specifically, this is the set of real symmetric matrices $\Omega$
that have $0$ entries
corresponding to non-edges of $G$,
have the all-ones vector in the
kernel, rank $n - d -1$, and the following additional property:
For each vertex $i\in V(G)$, there are neighbors
$\{i_1, \ldots, i_d\}$ of $i$ so that the columns of $\Omega$
corresponding to the vertices in
\[
    V(G) \setminus \{i, i_1, \ldots, i_d\}
\]
are linearly independent.
\end{definition}
\begin{lemma}\label{lem: pseudo-euc gram}
Suppose that $\v_1, \ldots, \v_n$ are vectors that span $\RR^D$,
and let $\Psi$ be the Gram matrix of the $\v_i$ under any pseudo-Euclidean
bilinear form.  Then the columns of $\Psi$ corresponding to any spanning
subset of the $\v_i$ are independent.
\end{lemma}
\begin{proof}
Suppose that $B = \{\v_{i_1}, \ldots, \v_{i_D}\}$ are independent, and hence
span $\RR^D$.  Then the $D\times D$ Gram matrix $\Psi'$
of these vectors represents
the bilinear form in the basis $B$.  Since any pseudo-Euclidean form is non-degenerate 
(there is no vector orthogonal to all vectors), $\Psi'$
is non-singular.  Since $\Psi'$ appears as a sub-matrix of $\Psi$,
the associated columns are linearly independent in $\Psi$.
\end{proof}
As in the Gstress setting, the FStresses turn out to 
be exactly the ones arising from the FLSS construction.
\begin{lemma}\label{lem: fst iff flss}
The sets FLSS and FST are equal.
\end{lemma}
\begin{proof}
First we suppose that $\Omega$ is in FLSS.  All the
properties of FST except for independence of
complementary sets columns to a (subset of) each vertex
neighborhood follow in the same was as in the proof
of Lemma \ref{lem:pgeq}.  To establish the
last property, we use the defining properties of an
FOR and then Lemma \ref{lem: pseudo-euc gram} to show that
independence of spanning subsets of the FOR give
rise to independent columns in $\Omega$.

For the other direction, we can factor $\Omega$
as $X^tSX$, where $S$ is the diagonal $\pm 1$ matrix
giving the pseudo-Euclidean signature and $X$ is
$ n - d - 1\times n$ and real.  Lemma \ref{lem: pseudo-euc gram}
and the definition of a locally general position stress matrix
show that the vector configuration from the columns of $X$
have the properties of an FOR.
\end{proof}

At this point, we are ready to dualize back to frameworks. 
\begin{lemma}\label{lem: has FST stress iff Fstressable}
A framework is Fstressable  iff it has stress from $\FST$.
\end{lemma}
\begin{proof}
First suppose that $(G,\p)$ is a framework in dimension $d$ that 
is Fstressable.  Let $\Omega$ be a stress matrix for $(G,\p)$ 
of rank $n - d - 1$.  For each vertex $i$ of $G$, with 
neighbors $i_1, \ldots, i_k$, the points 
\[
  \p_i, \p_{i_1}, \ldots, \p_{i_k}
\]
are affinely spanning.  After relabeling, we may as well assume that 
\[
  \p_i, \p_{i_1}, \ldots, \p_{i_d}
\]
affinely span $\RR^d$.  By Gale duality, we conclude that the 
$n - d - 1$ columns of $\Omega$ corresponding to 
\[
  V(G)\setminus \{i, i_1, \ldots, i_d\}
\]
are linearly independent.  Hence $\Omega\in \FST$, as desired.

Now suppose that $(G,\p)$ is an affinely spanning framework in 
$\RR^d$ with a stress $\Omega\in \FST$.  By the definition 
of $\FST$, $(G,\p)$ has a stress of rank $n - d - 1$.  
Let a vertex $i\in V(G)$ be given.  By the definition of 
a locally full spanning stress, there are neighbors 
$i_1, \ldots, i_d$ of $i$ so that the columns of $\Omega$ 
corresponding to  
\[
  V(G)\setminus \{i, i_1, \ldots, i_d\}
\]
are linearly independent.  By Gale duality, the points 
\[
  \p_i, \p_{i_1}, \ldots, \p_{i_d}
\]
affinely span $\RR^d$.  As $i$ was arbitrary, $(G,\p)$ 
is Fstressable.
\end{proof}

And now we can show that, in a sense, Fstressable frameworks 
are degenerations of Gstressable frameworks.
\begin{theorem}
\label{thm:fapprox}
Let $\Omega\in \FST$ be a locally full spanning equilibrium stress of a 
graph $G$.  Then every standard topology neighborhood of $\Omega$ 
contains a Gstress.

As a consequence, $\FST$ is irreducible and of the same dimension
as $\GST$.
% Every stress from $\FST$ can be epsilon approximated by
% a stress from $\GST$. Thus $\FST$ is irreducible and of the same dimension
% as $\GST$.
\end{theorem}
\begin{proof}
% Let a standard topology neighborhood $U$ of $\Omega$ in the space of 
% symmetric matrices be given.
As $\Omega\in \FST$, by Lemma \ref{lem: fst iff flss}, 
$\Omega$ is the Gram matrix of a real or pseudo-Euclidean 
FOR $\v$ that has rank $D$ and is centered.  In particular,
the configuration matrix $V$ of $\v$ has the all ones
vector in its kernel. 

Let $\eps > 0$ be given.  By continuity of the map that 
sends a configuration to its Gram matrix, there is a 
$\delta > 0$ so that if $\v'$ is a configuration with 
\[
    \max_{1\le i\le n} \|\v_i - \v'_i\| < \delta
\]
then 
\[
    \|\Omega - \Omega'\| < \eps 
\]
where $\Omega'$ is the Gram matrix of $\v'$ and the norm 
on the lhs is the operator norm (or any equivalent to it).

By Lemma \ref{lem:lgor}, for any $r > 0$, there is a GOR 
$\v''$ so that 
\[
    \max_{1\le i\le n} \|\v_i - \v''_i\| < r
\]
Because $\v$ is a FOR and $\v''$ is a GOR, the configuration 
matrices $V$ and $V''$ have the same rank for any choice 
of $r$ and $\v''$.  If we let $x(r)$ be the orthogonal projection of 
the all ones vector onto the kernel of $V''$, then 
as $r\to 0$, $x(r)$ converges to the all ones vector.

It follows that we can find an $r > 0$ and a GOR $\v''$, 
so that 
\[
    \max_{1\le i\le n} \|\v_i - \v''_i\| < r < \delta/2
\]
and, for each coordinate $\alpha_i$ of $x(r)$, 
\[
    |1-1/\alpha_i|\|\v''_i\| < 2|1-1/\alpha_i|\|\v_i\| < \delta/2
\]

We now set, for each $1\le i \le n$
\[
    \v'_i = \frac{1}{\alpha_i}\v''_i
\]
This is possible for sufficiently small $\delta$ because it forces all of the $\alpha_i$ 
to be non-zero.  By construction 
\[
    \alpha_i \v''_1 + \cdots + \alpha_n \v''_n = 0
\]
so 
\[
    \v'_1 + \cdots + \v'_n = 0
\]
which implies that $\v'$ is centered.  Since none of the 
$\alpha_i$ are zero, by \cite[Lemma 3.2]{Alfakih-conn}, 
$\v'$ will also be a GOR.  Since  
\[
    \|\v_i - \v'_i\| \le \|\v_i - \v''_i\| + |1 - 1/\alpha_i|\|\v''_i\| < \delta 
\]
we have proven that $\GST$ is Euclidean-topology dense in $\FST$.

For the second statement, we simply note that as 
Gstresses are standard topology dense in the locally
full spanning stresses, they are also Zariski dense.
\end{proof}

Now we are ready to prove
Theorem~\ref{thm:fapprox2}.
The basic idea is to setup a map from stresses to kernel frameworks.
The only complication is that we do not have general position, so
we cannot fix the pinned vertices globally. Rather we do so
locally in a Zariski neighborhood of an appropriate $\Omega$.
\begin{proof}[Proof of Theorem~\ref{thm:fapprox2}]
Let $(G,\p)$ be an Fstressable framework, with
stress $\Omega$ from $\FST$
We select $d+1$ vertices that have a full span in $\p$ and pin them
in their places. Then we build a map from stress space to framework space
as in the proof of Proposition \ref{prop:gstConst}.
This rational map is continuous, and well defined over a Zariski-open 
neighborhood of $\Omega$.  By Theorem \ref{thm:fapprox}, this Zariski 
neighborhood contains has a Euclidean dense subset of Gstresses.
By continuity of the parameterization, we can, for any Euclidean 
neighborhood $U$ of $(G,\p)$, find a Gstress $\Omega'$ has a Gstressable 
kernel framework $(G,\p')$ in $U$.
\end{proof}

\begin{proof}[Proof of Theorem~\ref{thm:fapprox3}]
For the first statement, we will appply the
the same (complexified) strategy used to prove
Proposition \ref{prop:gstConst}.
We will select $d+1$ vertices
and pin to the cannonical simplex and
 use a rational map
from $\FST$ to Fstressable frameworks.

The only complication is that
in our current setting,
such a rational map may be undefined over
some subvariety $V$ of $\FST$, where, in the equilibrium framework,
the selected vertices would have a deficient span.
To deal with this
we can always pick a different set of vertices to pin, giving rise to a different
rational map, with it own undefined locus $W$.
Taking the union of the images over a finite number of
such maps we can obtain all Fstressable frameworks.

The irreducibility and
dimension  then follows from
the density of the Gstressable frameworks in
the Fstressable frameworks (Theorem~\ref{thm:fapprox2}).
\end{proof}

\section{Stressed Corank}
In this section we explore the typical rank of the rigidity matrix of a
Gstressable framework.  We will assume that $G$ is $d+1$-connected.

\begin{definition}
Let \defn{$\corank(G)$}
denote the the minimum of corank of the rigidity 
matrix $R(\p)$ over all 
 $d$-dimensional frameworks $(G,\p)$.
The quantity $\corank(G)$ is the corank of the rigidity matrix for 
every generic framework $(G,\p)$.
\end{definition}

\begin{definition}
Let $\corank(G,\p)$ be the corank of the rigidity matrix
for $(G,\p)$.
Let \defn{$\stressedCorank(G)$}
denote the minimum of $\corank(G,\p)$ over
all Gstressable frameworks.
\end{definition}
By definition $\stressedCorank(G) \ge \corank(G)$.

Because Gstressable is irreducible, we can talk about generic
frameworks and their behaviour.
\begin{lemma}
\label{lem:sco}
$\stressedCorank(G)$ is equal to the 
dimension of the space of equilibrium stresses for 
every generic Gstressable framework.
\end{lemma}
\begin{proof}
Since the stress count is the dimension of the left kernel of the rigidity matrix, 
and the rank of a matrix drops only when all its minors of some order vanish,  
the rank only drops on Zariski-closed subsets.
\end{proof}

The next statement gives us another way to think about stressedCorank
\begin{proposition}
Suppose $\Omega$ is generic in $\GST$. Let $\p$ be a kernel
framework of $\Omega$ with a $d$-dimensional span.
Then $\corank(G,\p) = \stressedCorank(G)$.
\end{proposition}
\begin{proof}
The basic principle is that the image of a generic
point of an irreducible algebraic object under
a rational map, defined over $\QQ$, is a generic point in the image.
For this proposition, we build a map that takes
a stress $\Omega$  in $\GST$
and a non-singular affine transform $A$
to a kernel framework $\p$.
The image of this map is the Gstressable frameworks
(Lemma~\ref{lem:dual}).

When $\Omega$ is generic in $\GST$, we can find
an $A$ so that $(\Omega,A)$ is generic as a point
of a bundle.
Then from Lemma~\ref{lem:sco},
$\corank(G,\p) = \stressedCorank(G)$.
\end{proof}

We are now in a position to relate the dimension of the 
set of Gstressable frameworks of a graph $G$ to the 
stressed corank of $G$.
\begin{proposition}
For every graph $G$ and dimension $d$:
\ba
\dim(Gstressable) - d(d+1) &=& m - \binom{d+1}{2} - stressedCorank(G)
\\
\dim(Gstressable)  &=& m +\binom{d+1}{2}
- stressedCorank(G)
\ea
\end{proposition}
\begin{proof}
In an irreducible algebraic setting, the dimension of the image of
a rational map plus the dimension of a generic fiber equals the dimension of
the domain (See e.g.~\cite[Theorem A.12]{ghave}).

On the first line, left hand side, we have the image of a map
from stresses to pinned frameworks. The
$d(d+1)$ accounts for the $d$-dimensional affine transforms.

\end{proof}

\subsection{Characterizing stressedCorank}

\begin{lemma}
\label{lem:Gnecc}
If $G$ is generically globally rigid in $\RR^d$, then
any  generic framework $(G,\p)$ in $\RR^d$
must have an equilibrium stress matrix  in $\GST$, and thus be Gstressable.
\end{lemma}
\begin{proof}
Let $\p$ be generic.  By Theorem~\ref{thm:necc}
$(G,\p)$ must be in the kernel of some stress matrix
$\Omega$  of rank $n - d - 1$.  As $\p$ is generic, it
is in affine general position. Lemma \ref{lem:dual} then tells us
that $\Omega\in \GST$.
\end{proof}

\begin{theorem}
Suppose that $G$ is generically globally rigid in $\RR^d$.
Then $stressedCorank(G)=corank(G)$.
\end{theorem}
\begin{proof}
Let $\p$ be a generic framework. We have $corank(G,\p)=corank(G)$.
From Lemma~\ref{lem:Gnecc},
$\p$   is Gstressable.
Since $\p$ is a generic framework, $\p$ is also generic in Gstressable.
Thus, using Lemma~\ref{lem:sco}
$corank(G,\p)=stressedCorank(G)$.
\end{proof}

A much more interesting question arises when $G$ is
$(d+1)$-connected but not generically globally rigid.

In the next proof, we will use the following.  See \cite[Lemma B.4, inter alia]{ghave}
\begin{lemma}\label{lem:stress-wiggle}
Let $(G,\p)$ be a framework in dimension $d$, so that the rank of the rigidity matrix at $\p$ is
maximum over all configurations.  If there is a stress matrix
$\Omega$ for $(G,\p)$ that is  of rank $n - d - 1$, then every sufficiently
small neighborhood $U$ of $\p$ contains only frameworks $\q$ that
have a  stress matrix of rank $n - d - 1$.  (Moreover, if $\Omega$ is
PSD, then all the frameworks in $U$ have a PSD stress matrix of
rank $n - d -1$.)
\end{lemma}

\begin{proposition}
Suppose that $G$ is
$(d+1)$-connected but not generically globally rigid.
Then $stressedCorank(G) >  corank(G)$.
\end{proposition}
Before giving the proof, we observe the following easy case.  Suppose that $G$ is generically stress free.  Certainly
$stressedCorank(G) > 0 = corank(G)$.
\begin{proof}
Suppose there were a Gstressable framework $\p$
with only $corank(G)$ stresses. This means that its rigidity
matrix displays the maximal possible rank for $G$ and $d$.
Then from Lemma~\ref{lem:stress-wiggle},
all configurations $\q$ in a sufficiently small neighborhood of $\p$
have a stress matrix of rank $n-d-1$.
In this neighborhood, there must be a generic $\q$, where the stress matrix would
give a certificate of generic global rigidity as per Theorem~\ref{thm:suff}.
A contradiction.
\end{proof}

One might be tempted to expect that for a non
generically globally rigid graph, we would have
$stressedCorank(G)=corank(G)+1$.
But this is not true. In the example of Figure~\ref{stressed-rank-jump.fig},
we have $corank(G)=0$, but it can be shown that
$stressedCorank(G)=2$. So this presents the following question: Is there a characterization that
gives us stressedCorank(G)?
Perhaps the simplest relevant case is for $G$
a non generically globally rigid,
$3$-connected, planar graph with at least one triangular face.

\begin{figure}[h]
\centering
\includegraphics[scale=0.35]{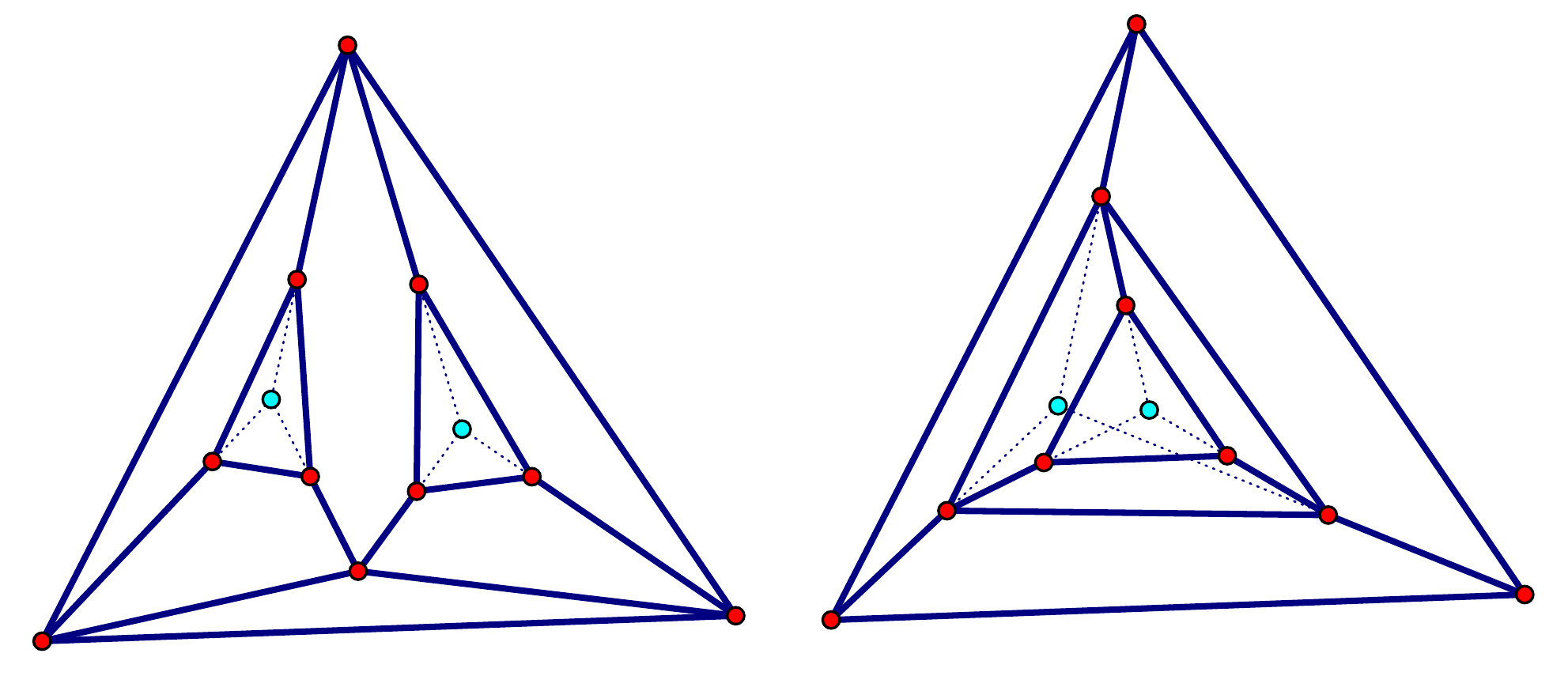}
\captionsetup{labelsep=colon,margin=1.3cm}
\caption{This shows two frameworks such that when the configuration is generic, they only have the $0$ stress and they are infinitesimally rigid, so by the count of edges and vertices, the number of independent stresses is equal to the number of independent infinitesimal flexes.  But when the
stress coefficients on all the inner members is chosen generically, say with positive numbers, and the
stress coefficients on the external triangle are chosen to create equilibrium (necessarily with negative stresses), then the dimension of the stress space is two.  The green vertices are not vertices of the framework, but the intersection of three lines that are extensions of edges in the framework needed to insure the equilibrium condition for our assumed stress. In both cases there are two independent infinitesimal flexes that are infinitesimal rotations, each on a triangle about its center green point.}
\label{stressed-rank-jump.fig}
\end{figure}

\section{Generic universal rigid framework}
\label{sec:stress}
The following is the main theorem of~\cite{ghave}.
\begin{theorem}
\label{thm:main}
If $G$ is generically globally rigid in $\RR^d$, then
there exists a framework $(G,\p)$ in $\RR^d$ that is
infinitesimally rigid in $\RR^d$ and super stable.
Moreover, every framework in a small enough neighborhood of $(G,\p)$
will be infinitesimally rigid in $\RR^d$ and
super stable. This neighborhood  must include some generic framework.
\end{theorem}
Here we give a much simpler proof based on $\GST$.

\begin{proof}[Proof of Theorem~\ref{thm:main}]
Let GC be the configurations that are in general
position. This is quasi-projective and irreducible.
Let GIF be the
general position frameworks that are infinitesimally flexible. This is a Zariski-closed subset of GC
(ie. defined by the vanishing of a non-trivial
polynomial).
Let $\GST(GIF)$ be the stresses in $\GST$ that have
a kernel framework in GIF (and so no kernel frameworks
that are infinitesimally rigid).
By
Lemma~\ref{lem:Gnecc},  $\GST(GIF) \subsetneq \GST$.
Now we will show that $\dim(\GST(GIF))<\dim(\GST)$.

We (again) consider our  rational map that maps from a matrix
$\Omega \in \GST$
to the
framework in its kernel with a $d$-dimensional affine span
and with $d+1$ chosen vertices pinned in
canonical locations.
Due to general position,
this map is defined everywhere over $\GST$.
Let $W$ be the preimage of GIF under this map.
This is a Zariski closed subset of $\GST$.
We have $W \subseteq \GST(GIF) \subsetneq \GST$.
Since $GIF$ is a Zariski-closed subset of the range,
its preimage, $W$, lies
in a Zariski closed strict subset of $\GST$.

From Theorem~\ref{thm:hen}, $G$ must be $(d+1)$-connected.
Thus from Theorem~\ref{thm:gstr}, Gstr must be
irreducible and of dimension $D_L$.
Since this Zariski closed
subset $W$ is strict and $\GST$ is irreducible, then
it must be of strictly lower dimension.

Now, any framework in $GIF$ is related, by a non-singular
affine transform to framework with its
$d+1$ chosen vertices
in the canonical position. This canonically transformed framework
must also be in $GIF$, as infintesimal flexibility is invariant to non-singular
affine transforms.  Thus
$\GST(GIF)=W$.
So we have shown $\dim(\GST(GIF))<D_L$.

Since $\GST(GIF)$ has dimension less than $D_L$,
and $\PGST$ is dimension $D_L$ by Corollary~\ref{cor:dl},
we must have $\GST(GIF) \not\supset \PGST$.
Any stress in $\PGST$ and not in
$\GST(GIF)$ must have its full-span
kernel frameworks be infintesimally rigid.
This means we must have a PSD max rank stress matrix $\Omega$
with a infinitesimally rigid kernel
 $(G,\p)$.
 This framework must be super stable.

% The framework must be in affine general position
% from Lemma~\ref{lem:dual},
% and thus super stable
% from Theorem~\ref{thm:alfss}.

To finish the proof, we note
infinitesimal rigidity is an open property  will
preserved in a small enough neighborhood of $\p$.
% the following sets of configurations are all open and dense: IR frameworks and
% frameworks in affine general position.
% Any sufficiently small neighborhood of $\p$
% will preserve these properties.
Lemma \ref{lem:stress-wiggle}, which applies by infinitesimal rigidity of $(G,\p)$
then gives us a neighborhood $U$ of $\p$ in which every framework $(G,\q)$
has a PSD stress matrix of rank $n - d -1$.
This gives us super stability in this neighborhood.

Finally, any open neighborhood contains a generic point.

\end{proof}

%%%%%%%%%%%%%%%%
%\newpage
\appendix{}
%%%%%%%%%%%%%%%%

%%%%%%%%%%%%%%%%
\section{Algebraic geometry background}
%%%%%%%%%%%%%%%%

We quickly review some algebraic geometry.

\begin{definition}
  \label{def:generic-gen}
  An real
\defn{algebraic set}~$V$
is a subset of $\RR^n$
that can be
  defined by a finite set of algebraic equations.

A \defn{Zariski open subset}
of some set $S \in \RR^n$
is a subset  of $S$ defined by removing an
algebraic set from $S$.
A \defn{Zariski closed subset} of some set $S$ is
the intersection of $S$ and an
algebraic set.

A \defn{quasi-projective variety} is a Zariski open
subset of an algebraic set. (Ie. we can cut
out some subvariety of a variety).

A \defn{constructible set} is the finite union of
quasi-projective varieties.

A \defn{semi-algebraic set}~$S$
  is a subset of $\RR^n$ that can be defined by
 a finite set of   algebraic
  equalities and inequalities
as well as
a finite number of Boolean operations.

Any semi-algebraic set is constructible.
Any constructible set is quasi-projective.
Any quasi-projective variety is algebraic.

A real semi-algebraic set $S$ has a
real \defn{dimension}
$\dim(S)$, which we will define as the largest $t$ for which there
is an open subset of~$S$, in the Euclidean topology, that is
is a smooth $t$ dimensional smooth sub-manifold of $\RR^n$.

The Zariski-closure of a semi-algebraic set $S$
is the smallest
real algebraic set $V$ containing $S$. $V$ will have the
same dimension as $S$.

An algebraic set called \defn{irreducible}
if it is not the union of two proper
 algebraic subsets.
%  defined over $\RR$.

A semi-algebraic set is called irreducible if its
Zariski closure is irreducible.

Any Zariski-closed strict subset of an irreducible
semi-algebraic set must be of
strictly lower dimension.

The image of a  real algebraic, quasi-projective,
constructible, or
semi-algebraic  set
set under a polynomial or rational  map
is semi-algebraic.
If the domain is irreducible, then so too is the image.
\end{definition}

\begin{definition}
In the complex setting $\CC^n$, we can also
define algebraic sets, quasi-projective varieties and
constructible sets accordingly. (There is no notion
of semi-algebraic sets.)

The image of a  complex algebraic, quasi-projective,
or constructible set
set under a polynomial or rational  map
is constructible.
If the domain is irreducible, then so too is the image.

\end{definition}

\begin{definition}
  Given an irreducible object $S$ of any one of the above types,  we call a point $\x \in S$ \defn{generic}
  if $\x$ does not satisfy any polynomial equations with
  coefficients in $\QQ$ that do not vanish on all of $S$.

  A rational map acting on $S$, and defined using coefficients
  in $\QQ$,
  will map generic points of the domain
  to generic points in the image.
\end{definition}

\begin{lemma}
\label{lem:locus}
If the real locus of an irreducible complex
quasi-projective variety  has a real dimension matching that
of the complex quasi-projective variety, then it too must be irreducible.
\end{lemma}
\begin{proof}[Proof Sketch]
Let $V$ be a real algebraic set and $V^*$ be the smallest
complex algebraic set that contains $V$ (called
its complexification).
The complex dimension of $V^*$ must equal
the real dimension of $V$~\cite{whitney}.
If $V$ is reducible, then so too is
$V^*$~\cite{whitney}.

Now suppose that $W$ is an irreducible complex
algebraic set and $V$ its real locus.
$V$ is a real algebraic set with real dimension less than or equal to
the complex dimension of $W$.
Suppose that the real dimension of $V$ equals the
complex dimension of $W$. Then $V^*$ has the
same dimension of $W$. By construction, $V^*$
must be contained in $W$. Since $W$ is irreducible,
any \emph{strict}
algebraic subset must be of lower dimension,
so in fact $W=V^*$. From the previous paragraph,
if $V$ were reducible, then so too would be
$V^*$ and $W$, forming a contradiction. Thus $V$ must be irreducible.

We can apply the same argument to the quasi-projective setting.
We start with $T$, a complex quasi-projective variety, and write its
Zariski closure as $\overline{T}$.
(The Zariski closure of a quasi-projective variety does not affect dimension or reducibility.)
We let $S$ be the real locus of
$T$ and let $\overline{S}^*$  be the smallest complex algebraic set containing $S$.
Again, by construction $\overline{S}^* \subseteq \overline{T}$.
Then we can prove, using dimension and irreducibility as above,  that $\overline{S}^*=\overline{T}$.
Again, reducibility of $S$ would imply reducibility for $\overline{S}^*$ and therefore for $T$.
\end{proof}

\section{Statics}
In this appendix we state some highly specific statics facts.  See
a reference like \cite{williams-tensegrity} or \cite{ivan-statics}.

A \defn{load} $\f$ on a framework $(G,\p)$ is an assignment of vectors $\f_i$ to
each vertex $i\in V(G)$.  A load $\f$ is said to be \defn{resolvable} by $(G,\p)$
if there are numbers $\rho_{ij}$ so that, at each vertex $i$
\[
    \sum_{j\in N(i)} \rho_{ij}(\p_j - \p_i) = -\f_i
\]
An \defn{equilibrium} load is any load resolvable by $(K_n,\p)$.
A framework is called \defn{statically rigid} if every
equilibrium load is resolvable.  The important classical fact is.
\begin{lemma}\label{lem: static inf}
A framework $(G,\p)$ is statically rigid if and only if it is
kinematically rigid.  If $(G,\p)$ is minimally statically rigid,
there is a unique resolution for every equilibrium load.
\end{lemma}
\begin{proof}
This follows from linear algebra duality.  The row rank of the
rigidity matrix is the same as the dimension of the space of
resolvable equilibrium loads.  For $K_n$, this is $dn - \binom{d+1}{2}$,
so if $(G,\p)$ is statically rigid, its rigidity matrix has row rank, and
hence column rank this large.
\end{proof}

It's useful to have a different characterization of
equilibrium loads, which is more extrinsic.  Informally,
it says there is no net force and no net torque.
\begin{lemma}\label{lem: exeqload}
A load $\f$ on $(K_n,\p)$is an equilibrium load if and only if
\[
    \sum_{i=1}^n \f_i = 0\qquad
    \text{and}
    \qquad
    \sum_{i=1}^n \f_i \wedge \p_i = 0
\]
\end{lemma}
\begin{proof}
First suppose that $\f$ is an equilibrium load.  Then we have
for appropriate edge weights $\rho_{ij} = \rho_{ji}$
\[
    -\sum_{i=1}^n \f_i =
    - \sum_{i=1}^n \sum_{j\neq i} \rho_{ij}(\p_j - \p_i)
    = -\sum_{i < j} \rho_{ij}[(\p_j - \p_i) + (\p_i - \p_j)]
    = 0
\]
and also
\[
    \sum_{i=1}^n \f_i\wedge \p_i
    =
    \sum_{i=1}^n
    \left(
        \sum_{j\neq i} \rho_{ij}(\p_j - \p_i)
    \right)\wedge \p_i
    =
    \sum_{i = 1}^n \sum_{j\neq i} \rho_{ij}(\p_j\wedge \p_i)
    = \sum_{i < j}\rho_{ij}(\p_i \wedge \p_j + \p_j \wedge \p_i)
    = 0
\]

For the other direction we use linear algebra duality.
The equilibrium loads
are the image of the transposed rigidity matrix $R^t(\p)$.
If $\p$ has $d$-dimensional affine span, this space has dimension
$dn - \binom{d+1}{2}$, since complete graphs are always infinitesimally
rigid if spanning.  On the other hand, the two equations in the
statement of the lemma impose $d + \binom{d}{2}$ linear constraints
on a possible load $\f$.  Hence the spaces are the same.
\end{proof}
Immediately, we get.
\begin{lemma}\label{lem: eq load support}
Suppose that $\f$ is an equilibrium load on a framework $(G,\p)$
that is supported only on $S\subseteq V(G)$.  Then $\f$
is resolvable by $(K_{|S|},\p_S)$.
\end{lemma}
\begin{proof}
Throwing away zero vectors doesn't change the constraints in
Lemma \ref{lem: exeqload}.
\end{proof}

% If $G = (V,E)$ is a graph, we say that $(X,Y,S)$
% is a \defn{separation} if $V = X\cup Y$, $S = X\cap Y$ and there are no
% edges between $X\setminus Y$ and $Y\setminus X$.  The \defn{size} of a separation
% is the size of $S$.
% \begin{lemma}\label{lem: stress sep}
% Let $(G,\p)$ be an affinely spanning framework in $\RR^d$.  Suppose that $\omega$ is an
% equilibrium stress of $(G,\p)$ and that $(X,Y,S)$ is a separation of $G$.
% Then, the restriction of $\omega$ to $X$ or $Y$ induces an equilibrium load $\f_S$ on $\p_S$, which is,
% in particular resolvable by $(K_{|S|},\p_S)$.  This implies that all the vectors in $\f_S$
% lie in the affine span of $\p_S$.
% \end{lemma}
% \begin{proof}
% The restriction of $\omega$ to $X$ produces an equilibrium load on $(G,\p)$ that is
% zero everywhere except on the vertices of $S$.  Now we are done by
% Lemma \ref{lem: eq load support}.
% \end{proof}

%%%%%%%%%%%%%%%%
\def\v{\oldv}
\bibliographystyle{abbrvlst}
\bibliography{framework}
%%%%%%%%%%%%%%%%

\end{document}